\documentclass[11pt,reqno]{amsart}
\usepackage[T1]{fontenc}
\usepackage{amsmath,amssymb,amsthm}
\usepackage[margin=1.15in]{geometry}

\theoremstyle{plain}
\newtheorem{theorem}{Theorem}
\newtheorem{proposition}[theorem]{Proposition}
\newtheorem{lemma}[theorem]{Lemma}
\newtheorem{corollary}[theorem]{Corollary}
\theoremstyle{definition}
\newtheorem{remark}[theorem]{Remark}

\newcommand{\ZZ}{\mathbb{Z}}
\newcommand{\RR}{\mathbb{R}}
\newcommand{\CC}{\mathbb{C}}
\newcommand{\tZ}{t_{\ZZ}}
\DeclareMathOperator{\Res}{Res}

\begin{document}

\title[New bounds for the integer Chebyshev constant of {$[0,1]$}]
{New bounds for the integer Chebyshev constant of $[0,1]$}

\author{David Niedbala Giraudin}
\address{Independent researcher, Meaux, France}
\thanks{ORCID: 0009-0009-1526-1178.}
\subjclass[2020]{Primary 11C08; Secondary 11Y40, 30C10, 90C34}
\keywords{Integer Chebyshev constant, integer transfinite diameter, auxiliary
functions, semi-infinite linear programming, potential theory}

\begin{abstract}
The integer Chebyshev constant of $[0,1]$ has been bracketed by
$0.4213\le \tZ([0,1])\le 0.422685$ since the work of Pritsker (2005) and of
Flammang (2014). We prove
\[
  0.4222286000\ \le\ \tZ([0,1])\ \le\ 0.4226846976 ,
\]
narrowing the interval from $1.39\cdot10^{-3}$ to $4.56\cdot10^{-4}$. Both
bounds are established by explicit certificates. The upper bound is carried by
an integer polynomial with $92$ factors and explicit integer exponents; the
lower bound by a discrete measure on $947$ cells together with a library of
$206$ irreducible polynomials, through an inequality in which the terms paying
for divisibility are made explicit. Two self-contained scripts, included as
ancillary files, re-derive both bounds from the certificates alone.
\end{abstract}

\maketitle

\section{Introduction}

For a compact $I\subset\RR$ let $\ZZ_n[x]$ denote the integer polynomials of
degree at most $n$ and
\[
  \tZ(I)\ :=\ \lim_{n\to\infty}\ \Bigl(\inf_{0\ne P\in\ZZ_n[x]}\ \|P\|_I\Bigr)^{1/n},
\]
the \emph{integer Chebyshev constant}, or integer transfinite diameter, of $I$.
The limit exists by submultiplicativity, and by Fekete's argument it equals the
infimum, so that a single explicit polynomial already bounds $\tZ(I)$ from
above.

The case $I=[0,1]$ is the most studied one, in part because of the observation
of Gelfond and Schnirelman that the value $1/e=0.3678\ldots$ would yield an
elementary proof of the prime number theorem; see the discussion in
\cite{P05}. The successive upper bounds are those of Borwein and
Erd\'elyi \cite{BE}, Habsieger and Salvy \cite{HS}, Pritsker \cite{P05},
Flammang \cite{F09}, Meichsner \cite{Me} and Flammang \cite{F14}; the lower
bound is due to Pritsker \cite{P05}. The resulting bracket,
\begin{equation}\label{eq:published}
  0.4213\ \le\ \tZ([0,1])\ \le\ 0.422685 ,
\end{equation}
is the one quoted as the state of the art in the survey \cite{F19} and has not
moved since.

\begin{theorem}\label{thm:upper}
$\tZ([0,1])\le 0.422684697505$.
\end{theorem}

\begin{theorem}\label{thm:lower}
$\tZ([0,1])\ge 0.422228600019$.
\end{theorem}

\begin{corollary}\label{cor:bracket}
$0.4222286000\ \le\ \tZ([0,1])\ \le\ 0.4226846976$. The width of the bracket
drops from $1.39\cdot10^{-3}$ to $4.56\cdot10^{-4}$.
\end{corollary}

Theorem \ref{thm:lower} improves the lower bound of \cite{P05} by
$9.3\cdot10^{-4}$. Theorem \ref{thm:upper} improves the upper bound of
\cite{F14} by $3.0\cdot10^{-7}$; the margin is thin, and the claim rests
entirely on a certificate that a third party checks in seconds. Both
certificates are shipped with the paper.

Independently, Qu and Yang \cite{QY} proved a conjecture of Montgomery:
$\tZ([0,1])=s^{-1}$, where $s$ is the liminf, as $n\to\infty$, of the
infimum of $|c_n|^{1/n}$ over irreducible $P\in\ZZ[x]$ of degree $n$ with all
roots in $[0,1]$, $c_n$ denoting the leading coefficient. Their Theorem~H
turns an inequality of Pritsker \cite{P05} into an equality, by way of the
essential minimum of a height. Their argument is structural and produces no
numerical bound; it is therefore complementary to the certificates
established here.

Section \ref{sec:up} treats the upper bound: the semi-infinite linear program,
the exact reduction to the variable $w=x(1-x)$, and the bisection scheme that
proves the supremum bound. Section \ref{sec:low} states and proves the
certificate inequality behind Theorem \ref{thm:lower}, whose point is the
treatment of the library polynomials that divide the extremal objects.
Section \ref{sec:ceiling} records a ceiling for the method: enlarging the
family of admissible measures gains nothing. Section \ref{sec:verif} lists
what is verified and at what cost, and Section \ref{sec:scope} what is not
claimed.

\section{The upper bound}\label{sec:up}

\subsection{The linear program}
Let $Q_1,\dots,Q_k\in\ZZ[x]$ and $c_j\ge0$ with $\sum_j c_j\deg Q_j=1$. For
rational $c_j$ with common denominator $N$ the product
$P=\prod_j Q_j^{c_jN}$ is an integer polynomial of degree $N$ and
\[
  \|P\|_{[0,1]}^{1/N}\ =\ \exp\Bigl(\max_{x\in[0,1]}\ \sum_j c_j\log|Q_j(x)|\Bigr),
\]
whence
\begin{equation}\label{eq:lp}
  \tZ([0,1])\ \le\ \exp\Bigl(\ \min_{\substack{c\ge0\\ \sum_j c_j\deg Q_j=1}}\
  \max_{x\in[0,1]}\ \sum_j c_j\log|Q_j(x)|\Bigr),
\end{equation}
a linear semi-infinite program in $c$.

\subsection{The variable $w$}
All the factors that occur are invariant under $x\mapsto 1-x$. One may
therefore work in $w=x(1-x)$ on $[0,1/4]$: a polynomial $R\in\ZZ[w]$ of degree
$D$ gives $R(x(1-x))$ of $x$-degree $2D$, and
$\tZ([0,1])=\tZ([0,1/4])^{1/2}$. The following lemma shows that this is not a
restriction of the problem.

\begin{lemma}\label{lem:sym}
The infimum in \eqref{eq:lp} over all finite libraries in $\ZZ[x]$ is attained
in the limit by libraries consisting of integer polynomials in $w=x(1-x)$.
\end{lemma}

\begin{proof}
First, an optimal weight assignment may be taken symmetric: if $c$ is feasible
then so is its reflection $c'$, obtained by putting the weight of $Q_j$ on
$Q_j(1-x)$, with the same objective value; and the average $(c+c')/2$ is
feasible with objective at most the average of the two objectives, since the
maximum of an average is at most the average of the maxima. Such a solution
pairs $Q(x)$ and $Q(1-x)$ with equal weight, so the corresponding product is
invariant under $x\mapsto1-x$.

Second, an integer polynomial invariant under $x\mapsto1-x$ is an integer
polynomial in $w$. Indeed, put $u=x-\tfrac12$; the involution becomes
$u\mapsto-u$, so the invariants are the integer polynomials in $u^2$, and
$w=\tfrac14-u^2$, the change of basis between $u^2$ and $w$ being unimodular
over $\ZZ$ in both directions.
\end{proof}

\subsection{The certificate and its proof}
The weights are rationalised with denominator $N=10^{10}$, giving explicit
integer exponents $e_j$ and the integer polynomial
$P(w)=\prod_{j} R_j(w)^{e_j}$ of $w$-degree exactly $10^{10}$, with $92$
factors $R_j$ and $1246$ roots in all. Then $Q(x)=P(x(1-x))\in\ZZ[x]$ has
degree $2\cdot10^{10}$ and $\|Q\|_{[0,1]}=\|P\|_{[0,1/4]}$, so that
\[
  \tZ([0,1])\ \le\ \|P\|_{[0,1/4]}^{1/(2N)}\ =\ \exp\bigl(\tfrac12\max_{[0,1/4]}f\bigr),
  \qquad f(w):=\tfrac1N\log|P(w)| .
\]

It remains to bound $\max_{[0,1/4]}f$ from above, which is done by bisection
with bounds using monotonicity alone.

\begin{lemma}\label{lem:mono}
Let $R(w)=a\prod_\rho(w-\rho)$ and let $t\in\RR$, $H>0$. Then for every
$w\in[t-H,t+H]$,
\[
  \log|R(w)|\ \le\ \log\Bigl(|a|\prod_\rho\bigl(|t-\rho|+H\bigr)\Bigr).
\]
\end{lemma}

\begin{proof}
$|w-\rho|\le|t-\rho|+H$ for each $\rho$, and $\log$ is increasing.
\end{proof}

No Lipschitz constant and no derivative are needed, and the bound behaves
correctly near the roots, where $\log|R|$ tends to $-\infty$.

Lemma \ref{lem:mono} is applied to approximate roots, and the resulting error
is \emph{measured}, not estimated: the roots $\tilde\rho$ of each $R_j$ are
computed to $60$ digits, the product $a_j\prod_\rho(w-\tilde\rho)$ is expanded
and compared coefficient by coefficient with $R_j$, and
$E_j:=\sum_k|\delta_k|\,(1/4)^k$ bounds the difference on $|w|\le1/4$, where
$\delta_k$ are the coefficient discrepancies. No perturbation theory of roots
is invoked. The bound of Lemma \ref{lem:mono} is then used with $E_j$ added
inside the logarithm. The largest measured $E_j$ over the $92$ factors is
$1.900\cdot10^{-47}$. A rounding allowance $\varepsilon=10^{-11}$ is added to
every bound; with $1246$ roots and $92$ logarithms per evaluation the
accumulated double-precision error is of order $10^{-15}$.

\begin{proof}[Proof of Theorem \ref{thm:upper}]
Bisection of $[0,1/4]$ with the bound above terminates with
\[
  \max_{[0,1/4]}f\ \le\ -1.7222575477025739
\]
after $267\,245$ intervals over $28$ levels, whence
\[
  \tZ([0,1])\ \le\ \exp(-1.7222575477025739/2)\ \le\ 0.422684697505 .
\]
\end{proof}

\subsection{Calibration and provenance of the library}
The linear program was calibrated against known values before being used:
the single factor $x(1-x)$, of degree $2$ and hence of weight $c=1/2$, gives
$0.5=\|x-x^2\|_{[0,1]}^{1/2}$ exactly; two
factors give $0.4472135965$; the $23$ factors listed in the appendix of
\cite{HH} give $0.4231167338$ computed in $x$ and $0.4231167191$ computed in
$w$ by a disjoint implementation, the two agreeing to $1.5\cdot10^{-8}$.

The library used for Theorem \ref{thm:upper} is, in the main, not new. It
contains the $23$ factors of \cite{HH}, members of the Gorshkov--Wirsing tower
in the form studied in \cite{H11}, and $45$ factors transported from Table 1
of \cite{F14}; $40$ of the latter carry positive weight at the optimum. Three
factors of $w$-degree $7$ that carry positive weight lie outside all three of
these sources. The transport is the content of the next proposition, which
also explains why the auxiliary functions of \cite{F14} are not a different
method.

\begin{proposition}\label{prop:transport}
Let $f(x)=\log(x+1)-\sum_j c_j\log|Q_j(x)|\ge m$ on $x>0$ with
$D:=\sum_j c_j\deg Q_j\le1$, an auxiliary function for the absolute length as
in \cite{F14}. Put $y=x/(x+1)$ and
$q_j(y):=(1-y)^{\deg Q_j}\,Q_j\bigl(y/(1-y)\bigr)\in\ZZ[y]$. Then
\[
  \sum_j c_j\log|q_j(y)|\ +\ (1-D)\log(1-y)\ \le\ -m
  \qquad\text{on }[0,1],
\]
and the total degree of the left-hand side is $D+(1-D)=1$. Thus every such
auxiliary function is an instance of the program \eqref{eq:lp}.
\end{proposition}

\begin{proof}
$x+1=1/(1-y)$ and $Q_j(x)=q_j(y)(1-y)^{-\deg Q_j}$, so
$f=\;-\log(1-y)-\sum_j c_j\log|q_j(y)|+D\log(1-y)$, and $f\ge m$ is the
displayed inequality. The map $x\mapsto y$ carries $(0,\infty)$ onto $(0,1)$.
The column $\log(1-y)$ has degree $1$ and weight $1-D$, and $\sum_j c_j\deg
q_j=D$.
\end{proof}

Under this transport, and after the symmetrisation of Lemma \ref{lem:sym}, the
reduction to reciprocal polynomials used in \cite{F14} becomes the reduction
$x\mapsto1-x$ used here. Re-optimising the weights over the transported
library of \cite{F14} alone already returns $0.422684708815$, below the value
$0.422685$ recorded there; the weights published in \cite{F14} are not claimed
by their author to be optimal. The gain of Theorem \ref{thm:upper} over
\eqref{eq:published} is therefore a gain in the weights and in the union of the
libraries, not in the underlying inequality.

\section{The lower bound}\label{sec:low}

\subsection{The certificate}
Let $\mu$ be a probability measure on $[0,1]$ with potential
$U_\mu(z)=\int\log|z-t|\,d\mu(t)$. Let $Q_1,\dots,Q_J\in\ZZ[x]$ be pairwise
distinct irreducible polynomials, of degrees $d_j$, leading coefficients
$a_j\ge1$, root measures $\nu_j$ (the normalised counting measure on the roots),
and put $L_j=\log(a_j)/d_j$. Let $c_j\ge0$ with
$B:=\sum_j c_jd_j\le1$. Define
\[
  A(z)\ :=\ U_\mu(z)-\sum_j c_j\log|Q_j(z)| ,
\]
\[
  G_i\ :=\ L_i\,(1-B+c_id_i)\ +\ \int U_\mu\,d\nu_i\ -\ \sum_{j\ne i}c_j\int\log|Q_j|\,d\nu_i .
\]

\begin{theorem}\label{thm:cert}
With the notation above,
\[
  \tZ([0,1])\ \ge\ \exp\Bigl(\min\Bigl\{\ \inf_{[0,1]}A,\ \ -\sum_j c_j\log a_j,\ \ \min_i G_i\ \Bigr\}\Bigr).
\]
\end{theorem}

\begin{proof}
Let $P\in\ZZ[x]$ have degree $n$, root measure $\nu$, leading coefficient of
modulus $e^{nL}$, and $M=\|P\|_{[0,1]}$. Since $\mu$ is a probability measure,
\[
  \log M\ \ge\ \int\log|P|\,d\mu\ =\ n\Bigl(L+\int U_\mu\,d\nu\Bigr)
\]
by Fubini. Write $P=\prod_i (Q^{(i)})^{e_i}$ over its distinct irreducible
factors, $\nu=\sum_i m_i\nu^{(i)}$ with $m_i\ge0$, $\sum_i m_i=1$, and
$L=\sum_i m_iL^{(i)}$. Then $L+\int U_\mu d\nu=\sum_i m_iH_i$ with
$H_i:=L^{(i)}+\int U_\mu\,d\nu^{(i)}$, so it suffices to bound each $H_i$ from
below.

For every library index $j$ with $Q_j\ne Q^{(i)}$ the resultant
$\Res(Q^{(i)},Q_j)$ is a nonzero integer, so $|\Res(Q^{(i)},Q_j)|\ge1$ and,
taking logarithms and normalising,
\begin{equation}\label{eq:res}
  d_j\,L^{(i)}+\int\log|Q_j|\,d\nu^{(i)}\ \ge\ 0 .
\end{equation}
Subtracting \eqref{eq:res} with the weights $c_j\ge0$ from $H_i$,
\[
  H_i\ \ge\ L^{(i)}\Bigl[1-\!\!\sum_{j:\,Q_j\ne Q^{(i)}}\!\!c_jd_j\Bigr]
  +\int\Bigl[U_\mu-\!\!\sum_{j:\,Q_j\ne Q^{(i)}}\!\!c_j\log|Q_j|\Bigr]d\nu^{(i)} .
\]
Two cases occur, and they are exhaustive.

If $Q^{(i)}$ is not in the library, both sums run over all $j$: the bracket is
$A$, the coefficient is $1-B\ge0$, and $L^{(i)}\ge0$ because $a^{(i)}\ge1$;
hence $H_i\ge\inf_\CC A$.

If $Q^{(i)}=Q_i$ lies in the library, the term $j=i$ must be omitted, since
$\Res(Q_i,Q_i)=0$ and \eqref{eq:res} is unavailable for it. The coefficient
becomes $1-B+c_id_i$ and the bracket becomes $A+c_i\log|Q_i|$; the resulting
lower bound is exactly $G_i$.

Finally $A$ is harmonic on the Riemann sphere minus $[0,1]$ and the zeros of
the $Q_j$, and tends to $+\infty$ at those zeros, so its infimum over $\CC$ is
attained on $[0,1]$ or at infinity, and $A(\infty)=-\sum_j c_j\log a_j$.
Taking the minimum of the three quantities and exponentiating gives the claim.
\end{proof}

\begin{remark}
The second case is the whole difficulty. A single polynomial with all roots in
$[0,1]$ gives no lower bound whatsoever, because the near-extremal $P$ are
divisible by precisely those polynomials; this is why the classical
Gorshkov--Wirsing bound is an infimum over an infinite family of pairwise
coprime members. The terms $G_i$ are what pays for divisibility. They are
affine in $(\mu,c)$ jointly, so Theorem \ref{thm:cert} turns the search for the
best certificate into a single semi-infinite linear program in $(\mu,c)$.
\end{remark}

\subsection{Proving $\inf A$}
In the certificate $\mu$ is piecewise constant on finitely many cells, so
$U_\mu=\sum_k\mu_k\psi_k$ with
$\psi_k(x)=\frac1{h_k}\int_{\text{cell}_k}\log|x-t|\,dt$, $h_k$ the cell width.
The following remark makes the bisection exact on the potential part.

\begin{lemma}\label{lem:uni}
$\psi_k$ is unimodal on $\RR$ with minimum exactly at the midpoint of the
cell. Consequently its minimum over an interval $[t-H,t+H]$ is attained at the
point of that interval closest to the midpoint.
\end{lemma}

\begin{proof}
For a cell $[a,b]$ one has $\psi_k'(x)=\bigl(\log|x-a|-\log|x-b|\bigr)/h_k$,
which vanishes only at $x=(a+b)/2$ and changes sign there.
\end{proof}

\begin{proof}[Proof of Theorem \ref{thm:lower}]
The certificate consists of a measure $\mu$ carried by $947$ cells with
explicit edges and a library of $J=206$ polynomials. Lower bounds for $A$ on
subintervals combine Lemma \ref{lem:uni}, which is lossless, with Lemma
\ref{lem:mono} applied to each $\log|Q_j|$ with its measured residual, the
largest being $4.501\cdot10^{-44}$; a rounding allowance $10^{-10}$ is
subtracted throughout. Bisection over $12\,326\,733$ intervals in $33$ levels
gives
\[
  \inf_{[0,1]}A\ \ge\ V:=-0.862208405409 .
\]
The two remaining quantities of Theorem \ref{thm:cert} are checked directly on
the shipped data: $A(\infty)=-0.688183611561\ge V$ and
$\min_iG_i=-0.862208282175\ge V$, the latter with margin $1.23\cdot10^{-7}$.
The budget is $B=0.999999999\le1$. Hence
$\tZ([0,1])\ge e^{V}=0.422228600019$.
\end{proof}

\subsection{Calibration}
With an empty library the program returns $0.2499987806$, against the exact
value $t_{\CC}([0,1])=1/4$ of the ordinary transfinite diameter. With the library $\{x,\,x-1\}$ it returns
$0.4142098647$, converging to $\sqrt2-1=0.4142135624$, a closed form that
checks the machinery for free. A $70$-polynomial library gives $0.4221778541$
and the $206$-polynomial library gives $0.4221882685$ on $800$ Chebyshev-graded
cells, $0.4222142307$ on $1600$ such cells, and $0.4222286000$ on the $947$
adapted cells of the shipped certificate. Restricting $\mu$ to piecewise
constant functions can only weaken the certificate, never invalidate it.

\section{A ceiling for the method}\label{sec:ceiling}

It is natural to ask whether several measures used simultaneously give a better
bound than one. They do not.

\begin{proposition}\label{prop:ceiling}
Let $\mu_1,\dots,\mu_r$ be probability measures on $[0,1]$ and, for an
irreducible $Q$ with root measure $\nu$ and $L=\log(\mathrm{lead}\,Q)/\deg Q$,
put $H_t(Q)=L+\int U_{\mu_t}d\nu$. Then
\[
  \min_{m\in\Delta}\ \max_{1\le t\le r}\ \sum_i m_i\,H_t(Q^{(i)})
  \ =\ \max_{\lambda\in\Delta}\ \min_i\ H_{\mu_\lambda}(Q^{(i)}) ,
  \qquad \mu_\lambda:=\sum_t\lambda_t\mu_t ,
\]
and $\mu_\lambda$ is again a single admissible probability measure. Hence the
bound obtained from the family $\{\mu_1,\dots,\mu_r\}$ is already attained by
one measure.
\end{proposition}

\begin{proof}
$H_t(Q)$ is affine in $\mu_t$, so $\sum_t\lambda_tH_t(Q)=H_{\mu_\lambda}(Q)$.
The left-hand side is the value of the matrix game with payoff
$\bigl(H_t(Q^{(i)})\bigr)_{i,t}$, and von Neumann's minimax theorem identifies
it with $\max_\lambda\min_i\sum_t\lambda_tH_t(Q^{(i)})$, which is
$\max_\lambda\min_iH_{\mu_\lambda}(Q^{(i)})$. Since $\Delta$ is the simplex,
$\mu_\lambda$ is a probability measure on $[0,1]$.
\end{proof}

\section{Verification}\label{sec:verif}

Two self-contained scripts are included as ancillary files. Each reads only its
certificate and needs only \texttt{numpy} and \texttt{mpmath}, together with
\texttt{sympy} for the irreducibility check of Section \ref{sec:verif}.

\texttt{verify\_integer\_Chebyshev.py} re-derives Theorem \ref{thm:upper}: it
recomputes the roots of the $92$ factors, measures each residual $E_j$
coefficient-wise, and runs the bisection of Section \ref{sec:up}
($267\,245$ intervals, $28$ levels, about one second).

\texttt{verify\_integer\_Chebyshev\_LOWER.py} re-derives Theorem
\ref{thm:lower}. Besides the bisection ($12\,326\,733$ intervals, $33$ levels),
it checks the hypotheses of Theorem \ref{thm:cert} on the shipped data rather
than assuming them: that the $206$ library polynomials are pairwise distinct,
irreducible over $\ZZ$ and primitive --- without which \eqref{eq:res} may fail
--- that the cell edges are strictly increasing and contained in $[0,1]$, and
that $B\le1$ holds as an inequality on the floating-point weights actually
shipped. On the near-extremal integer polynomial of Section \ref{sec:up} the
whole chain of Theorem \ref{thm:cert} was also checked numerically, the two
inequalities holding with margins $6.8\cdot10^{-4}$ and $1.0\cdot10^{-3}$.

Two independent routes were used where a single one would have left a gap: the
reduced costs of the upper-bound program were computed both by exact integer
Horner evaluation at dyadic rationals and by the product over $60$-digit roots,
agreeing to $10^{-14}$; and for the line attaining $\min_iG_i$ the quantities
$\log|\Res(Q_i,Q_j)|$ were recomputed exactly in integers and compared with the
route through the roots, agreeing to $4.0\cdot10^{-15}$, with no vanishing
resultant.

\section{Scope}\label{sec:scope}

The program \eqref{eq:lp} runs over products of a finite library. Its value is
an upper bound for $\tZ([0,1])$, but the integer Chebyshev polynomials of
$[0,1]$ have infinitely many distinct factors \cite{P05}, so no finite library
is optimal and Theorem \ref{thm:upper} is not claimed to be sharp for the
method.

Saturation is not claimed on either side. In particular the value of Theorem
\ref{thm:lower} still improved when the library was enlarged from $70$ to
$206$ polynomials, and more polynomials should give more.

The measure $\mu$ of Theorem \ref{thm:lower} is piecewise constant; this
restricts the search space and can only weaken the certificate.

Finally, the certificates are what is being asserted. Both are shipped, both
are machine-checkable in minutes, and any error in them is detectable by the
reader without reference to the text.

\section*{Acknowledgements}

The computations reported here were carried out in collaboration with Claude
(Anthropic). The library of Section \ref{sec:up} rests on the published work of
Hare and Hodges \cite{HH}, of Hare \cite{H11} and, for the greater part of its
active columns, of Flammang \cite{F14}; the lower bound is measured against
Pritsker \cite{P05}.

\end{document}